\documentclass{amsart}

\usepackage{amsmath}
\usepackage[applemac]{inputenc}
\usepackage{enumerate}
\usepackage{amssymb}\usepackage{epic}
\usepackage{eepic}
\usepackage{amssymb}
\usepackage{color}
\usepackage{datetime}

\usepackage{microtype}

\newcommand{\Cstar}{\ensuremath{{\mathrm C}^\ast}}

\newcommand{\cp}{\ell^1(G,A;\alpha)}

\theoremstyle{plain}

\newtheorem{theorem}{Theorem}[section]
\newtheorem{proposition}[theorem]{Proposition}
\newtheorem{lemma}[theorem]{Lemma}
\newtheorem{corollary}[theorem]{Corollary}

\theoremstyle{definition}

\newtheorem{remark}[theorem]{Remark}

\begin{document}

\title[Amenable crossed product Banach algebras]{Amenable crossed product Banach algebras associated with a class of $\boldsymbol{\Cstar}$-dynamical systems}

\author{Marcel de Jeu}
\address{Mathematical Institute,
Leiden University,
P.O. Box 9512,
2300 RA Leiden,
the Netherlands}
\email{mdejeu@math.leidenuniv.nl}

\author{Rachid El Harti}
\address{Department of Mathematics and Computer Sciences, Faculty of Sciences and Techniques,
University Hassan I, BP 577 Settat, Morocco}
\email{rachid.elharti@uhp.ac.ma
}

\author{Paulo R.\ Pinto}
\address{Instituto Superior T\'{e}cnico, University of Lisbon,
Av.\ Rovisco Pais 1, 1049-001 Lisbon, Portugal.}
\email{ppinto@math.tecnico.ulisboa.pt}


\subjclass[2010]{Primary 47L65; Secondary 43A07, 46H25, 46L55}

\keywords{Crossed product Banach algebra, \Cstar-dynamical system, amenable Banach algebra, amenable group}

\begin{abstract}
We prove that the crossed product Banach algebra $\cp$ that is associated with a \Cstar-dynamical system $(A,G,\alpha)$ is amenable if  $G$ is a discrete amenable group and $A$ is a commutative or finite dimensional \Cstar-algebra. Perspectives for further developments are indicated.
\end{abstract}

\maketitle

\section{Introduction}\label{sec:introduction}

A celebrated theorem of Johnson's \cite[Theorem~ 2.5]{johnson} shows that the locally compact Hausdorff topological group $G$ is amenable if and only if the Banach algebra $\mathrm{L}^1(G)$ is amenable. Various definitions of `amenability' for groups are in use, so let us mention explicitly that, in Johnson's definition (see~\cite[p.~32]{johnson}), the amenability of a locally compact Hausdorff topological group $G$ is to be understood as the existence of a left invariant mean on the space of bounded right uniformly continuous complex valued functions on $G$. By e.g.\ \cite[Definition~1.1.4 and Theorem~1.1.9]{runde}, this is equivalent to the existence of a left invariant mean on $\mathrm{L}^\infty(G)$, which is the definition of amenability that is used in the present paper as well as in \cite{pier} (see~\cite[Definition~4.2]{pier}), a source from which we shall use a few results.

According to \cite[Theorem~5.13]{dJMW}, $\mathrm{L}^1(G)$ is an example of a crossed product Banach algebra that can be associated with a general Banach algebra dynamical system $(A,G,\alpha)$ as in \cite[Definition~3.2]{DdJW}. More precisely: it is a member of a family of crossed product Banach algebras that can be associated with the \Cstar-algebra  dynamical system $(\mathbb C,G,\mathrm{triv})$, where the group acts as the identity on the \Cstar-algebra $\mathbb C$. Therefore, if $G$ is amenable, the amenable Banach algebra $\mathrm{L}^1(G)$ is a crossed product of the amenable group $G$ and the amenable \Cstar-algebra $\mathbb C$. Extrapolating this quite a bit, could it perhaps be the case that the `only if'-part of Johnson's theorem is a reflection of an underlying general principle, stating that, under appropriate additional conditions, crossed products of amenable locally compact Hausdorff topological groups and amenable \Cstar-algebras, associated with \Cstar-dynamical systems as in \cite[Definition~3.2]{DdJW}, are always amenable Banach algebras?

This paper is a first investigation into this matter. In support of the existence of such an underlying general principle, we shall establish that the crossed product Banach algebra $\cp$ (to be defined in Section~\ref{sec:main_result}) that is associated with a \Cstar-dynamical system $(A,G,\alpha)$ is amenable if  $G$ is a discrete amenable group and $A$ is a commutative or finite dimensional \Cstar-algebra.

Apart from this result and Johnson's theorem (the latter being also valid for non-discrete $G$), there is, in fact, some additional evidence that such a principle might exist. For this, we recall that the crossed product \Cstar-algebra $A\rtimes_\alpha G$ is amenable for every \Cstar-dynamical system $(A,G,\alpha)$ where $A$ is an amenable \Cstar-algebra and $G$ is an amenable locally compact Hausdorff topological group; see  \cite{gootman}, \cite[Theorem~7.18]{williams}, or \cite[Proposition~14]{green} for a proof of this statement. As we have formulated this result, we have used that amenability as a Banach algebra and nuclearity as a \Cstar-algebra are equivalent for \Cstar-algebras; see \cite{connes} and \cite{haagerup}. This formulation is deliberate, because \cite[Remark~9.4]{DdJW} shows that $A\rtimes_\alpha G$ is, in fact, a crossed product Banach algebra as in \cite[Definition~3.2]{DdJW}. Therefore, a principle as alluded to above could conceivably explain, from a result in Banach algebra theory, why these \Cstar-algebras $A\rtimes_\alpha G$ are amenable Banach algebras (and therefore nuclear \Cstar-algebras).

Let us note, however, that not all Banach algebras that are a crossed product of an amenable \Cstar-algebra and an amenable locally compact Hausdorff topological group as in \cite[Definition~3.2]{DdJW} are amenable. This already fails for the \Cstar-algebra $\mathbb C$ and the group $\mathbb Z$. Indeed, according to \cite[Theorem~5.13]{dJMW}, all Beurling algebras on $\mathbb Z$ figuring in \cite[Theorem~2.4]{badecurtisdales} are crossed product algebras of this type, but, according to the latter result, some of these algebras are amenable, whereas others are not. Hence additional conditions are necessary.

More research is needed to clarify the picture, and we have included a few thoughts on what might be nice to hope for in Section~\ref{sec:perspectives}. As mentioned, in the present paper we are concretely concerned with $\cp$, where $G$ is a discrete amenable group and $A$ is a commutative or finite dimensional \Cstar-algebra. The details for this  case are contained in the next section.

\section{Main result}\label{sec:main_result}

Suppose an action $\alpha: G\to \rm{Aut}(A)$ of the discrete amenable group $G$ as $^\ast$-automorphisms of the commutative or finite dimensional \Cstar-algebra $A$ is given, so that we have a \Cstar-dynamical system $(A,G,\alpha)$. Let us note that, if $A$ is commutative, $A$ is not required to be unital, and that, if $A$ is finite dimensional, $A$ is automatically unital (see \cite[Theorem~III.1.1]{davidson}). The aim is to show that the crossed product Banach algebra $\cp)$, that we are about to introduce, is amenable.

Before doing so, we ought to note, in view of the considerations in the Introduction, that, according to \cite[Theorem~5.13]{dJMW}, $\cp$ is, in fact, a crossed product Banach algebra as in \cite[Definition~3.2]{DdJW}. Furthermore, $G$ is obviously a locally compact Hausdorff topological group, and, according to \cite[p.~352]{wegge-olsen}, the \Cstar-algebras under consideration are amenable. Therefore, the amenability of $\cp$, as asserted in Theorem~\ref{thm:main_result}, fits into the general picture as sketched in the Introduction.

Returning to the main line, let us mention that such an algebra $\cp$ can be defined for every \Cstar-dynamical system $(A, G,\alpha)$ where $G$ is a discrete group, $A$ is an arbitrary \Cstar-algebra, and $\alpha: G\to {\rm Aut}(A)$ is an action of $G$ as $^\ast$-automorphisms of $A$. The definition, which is quite standard, is as follows. Let $\Vert \cdot \Vert $ denote the norm of the \Cstar-algebra $A$, and let
\begin{equation}\label{eq:definition_as_maps}
\cp=\{\,\,\texttt{a}: G\longrightarrow A:\ \Vert \texttt{a}\Vert:=\sum_{g\in G} \Vert a_g\Vert <\infty\,\},
\end{equation}
where, for typographical reasons, we have written $a_g$ for $\texttt{a}(g)$.
We supply $\cp$ with the usual twisted convolution product and involution, defined by
\begin{equation}\label{eq:prod}
(\texttt{a}\texttt{a}^\prime)(g)=\sum_{k\in G} a_k \cdot \alpha_k(a_{k^{-1}g}^\prime)\quad (g\in G, \,\texttt{a},\texttt{a}^\prime\in \cp)
\end{equation}
and
\begin{equation}\label{eq:involution}
\quad a_g^\ast=\alpha_g((a_{g^{-1}})^\ast)\quad (g\in G,\,\texttt{a}\in \cp),
\end{equation}
respectively, so that $\cp$ becomes a Banach algebra with isometric involution. If $A=\mathbb C$, then $\ell^1(G,\mathbb C;\textrm{triv})$ is the usual group algebra $\ell^1(G)$.
If $G=\mathbb Z$ and $A$ is a commutative unital \Cstar-algebra, hence of the form $\textrm{C}(X)$ for a compact Hausdorff space $X$, then there exists a homeomorphism of $X$ such that the $\mathbb Z$-action  is given by $\alpha_n(f)=f\circ \sigma^{-n}$ for all $f\in\textrm{C}(X)$ and $n\in\mathbb Z$.  The algebra $\ell^1(\mathbb Z,\textrm{C}(X); \alpha)$ has been studied in \cite{marcel0}, \cite{JTStudia}, and \cite{marcel1}, and the question whether this algebra is amenable was what originally led to this paper. According to Theorem~\ref{thm:main_result}, the answer is affirmative.

Returning to the case of an arbitrary \Cstar-algebra $A$, let us assume for the moment that $A$ is unital. A convenient way to work with $\cp$ is then provided by the following observations.

For $g\in G$, let $\delta_g: G\to A$ be defined by
$$\delta_g(k)=\left\{
\begin{array}{ll}
1_A & \hbox{if}\ k=g;\\
0_A & \hbox{if}\ k\not=g,
\end{array} \right.$$
where $1_A$ and $0_A$ denote the unit and the zero element of $A$, respectively. Then $\delta_g\in \cp$ and $\Vert \delta_g\Vert =1$ for all $g\in G$. Furthermore, $\cp$ is unital with $\delta_e$ as unit element, where $e$ denotes the identity element of $G$. Using \eqref{eq:prod}, one finds that \begin{equation*}
\delta_{gk}=\delta_{g}\cdot \delta_k \quad
\end{equation*}
for all $g,k\in G$. Hence, for all $g\in G$, $\delta_g$ is invertible in $\cp$, and, in fact, $\delta_g^{-1}=\delta_{g^{-1}}$. It is now obvious that the set $\{\,\delta_g : g\in G\,\}$ consists of norm one elements of $\cp$, and that it is a subgroup of the invertible elements of $\cp$ that is isomorphic to $G$.

In the same vein, it follows easily from \eqref{eq:prod} and \eqref{eq:involution} that we can view $A$ as a closed *-subalgebra of $\cp$, namely as $\{\,a\delta_e: a\in A\,\}$, where $a\delta_e$ is the element of $\cp$ that assumes the value $a\in A$ at $e\in G$, and the value $0_A\in A$ elsewhere.

If $\texttt{a}\in \cp$, then it is easy to see that $\texttt{a}=\sum_{g\in G} (a_g \delta_e)\delta_g$ as an absolutely convergent series in $\cp$. Hence, if we identify $a_g \delta_e$ and $a_g$, we have $\texttt{a}=\sum_{g\in G} a_g \delta_g$ as an absolutely convergent series in $\cp$.

Finally, let us note that an elementary computation, using the identifications just mentioned, shows that the identity
\begin{equation}\label{eq:conjugation}
\delta_g a\delta_g^{-1}=\alpha_g(a)
\end{equation}
holds in $\cp$ for all $g\in G$ and $a\in A$.

Note that $G$ acts on the unitary group of $A$ via $\alpha$, so that the semidirect product $U\rtimes_\alpha G$ can be defined. The following simple observation is the key to the proof of the main result. It exploits the fact that a unital \Cstar-algebra is spanned by its unitaries.

\begin{lemma}\label{lem:key}
Let $(A,G,\alpha)$ be a \Cstar-dynamical system, where $A$ is unital and $G$ is discrete. Let $U$ denote the unitary group of $A$. Then the set $\{\,u\delta_g : u\in U,\,g\in G\,\}$ is a subgroup of the invertible elements of $\cp$ that is canonically isomorphic to the semidirect product $U\rtimes_\alpha G$. This subgroup consists of norm one elements of $\cp$, and its closed linear span is $\cp$.
\end{lemma}

\begin{proof}
It is clear from the above that the given set is a subgroup of the invertible elements, and it follows easily from \eqref{eq:conjugation} that it is canonically isomorphic to $U\rtimes_\alpha G$. The statement on norm one elements is clear. Since every element of $A$ is a linear combination of four unitaries, the linear span of the subgroup contains $\{\, a\delta_g : a\in A,\,g\in G\,\}$. In view of the series representation of elements of $\cp$, the closed linear span of the subgroup is then the whole algebra.
\end{proof}

We shall also need the following.

\begin{lemma}\label{lem:reduction}
Let $(A,G,\alpha)$ be a \Cstar-dynamical system, where $A$ is non-unital and $G$ is discrete. Let $\alpha_1$ denote the unique extension of $\alpha$ to an action of $G$ as $^\ast$-automorphisms of the unitization $A_1$ of $A$. Then $\cp$ is canonically isometrically $^\ast$-isomorphic to a closed two-sided ideal of $\ell^1(G, A_1; \alpha_1)$ that contains a bounded two-sided approximate identity for itself.  Consequently, if $\ell^1(G, A_1; \alpha_1)$ is amenable, then so is $\cp$.
\end{lemma}

\begin{proof}
Viewing $A$ as a subset of $A_1$, it is obvious from \eqref{eq:definition_as_maps} and the additional definitions how $\cp$ can be viewed as an involutive Banach subalgebra of $\ell^1(G, A_1; \alpha_1)$. Since $A$ is a closed two-sided ideal of $A_1$, it is then clear from \eqref{eq:prod} that $\cp$ is a closed two-sided ideal of $\ell^1(G, A_1; \alpha_1)$. It can be identified with the set of all absolutely convergent series $\sum_{g\in G} a_g \delta_g$ in $\ell^1(G, A_1; \alpha_1)$ where all $a_g$ are elements of $A$. If $(e_\lambda)_{\lambda\in\Lambda}$ is a bounded two-sided approximate identity for $A$, then, using this series representation, it is easily seen that $(e_\lambda\delta_e)_{\lambda\in\Lambda}\subset \cp$ is a bounded two-sided approximate identity for $\cp$. The final statement therefore follows from \cite[Theorem 2.3.7]{runde}.
\end{proof}

As a final preparation for the proof of the main result, we mention the following cohomological characterization of amenability of a locally compact Hausdorff topological group. It can be found as \cite[Theorem 11.8.(ii)]{pier}; see also \cite[p.~17--18 and p.~99]{pier}. Note that neither the left, nor the right action of the identity element of the group is required to be the identity map.
\begin{proposition}\label{prop:cohomology}
Let $H$ be a locally compact Hausdorff topological group. Then the following are equivalent:
\begin{enumerate}
\item $H$ is amenable.
\item If $E$ is a Banach space that is a $H$-bimodule such that
\begin{enumerate}
\item for each $x\in E$, the maps $h\mapsto h\cdot x$ and $x\mapsto x\cdot h$ are continuous from $H$ into $E$ with its norm topology, and
\item there exists a constant $C$ with the property that $\Vert h\cdot x\Vert \leq C\Vert x\Vert $ and $\Vert x\cdot h\Vert \leq C\Vert x\Vert $ for all $x\in E$ and $h\in H$,
\end{enumerate}  and if $\phi: H\to E^\ast$ is map such that
\begin{enumerate}
\item[(c)] $\phi$ is weak$^\ast$-continuous,
\item[(d)] $\phi(h_1h_2)=h_1\cdot\phi(h_2)+\phi(h_1)\cdot h_2$ for all $h_1,h_2\in H$, and
\item[(e)] $\sup \{\,\Vert \phi(h)\Vert : h\in H\,\}<\infty$,
\end{enumerate}
then there exists $x^\ast\in E^\ast$ such that $\phi(h)=h\cdot x^\ast-x^\ast\cdot h$ for all $h\in H$.
\end{enumerate}
\end{proposition}

We can now establish our main result. We recall that a Banach algebra $B$ is amenable if every bounded derivation $D: B\to E^\ast$ from $B$ into the dual of any Banach $B$-bimodule $E$ is inner, i.e.\ if there exists $x^\ast\in E^\ast$ such that $D(b)=b\cdot x^\ast-x^\ast\cdot b$ for all $b\in B$.

\begin{theorem}\label{thm:main_result}
Let $(A,G,\alpha)$ be a \Cstar-dynamical system where $G$ is a discrete amenable group and $A$ is a commutative or finite dimensional \Cstar-algebra. Then $\cp$ is an amenable Banach algebra.
\end{theorem}

\begin{proof}
If $A$ is commutative, then Lemma~\ref{lem:reduction} shows that we may assume that $A$ is unital. Since, as we had already mentioned, finite dimensional $\Cstar$-algebras are always unital, we see that it is sufficient to prove the theorem for unital $A$.

Let $U$ denote the unitary group of $A$, and let $H=\{\,u\delta_g : u\in U,\, g\in G\,\}$. According to Lemma~\ref{lem:key}, $H$ and $U\rtimes_\alpha G$ are isomorphic as abstract groups; note that this implies that $H/U$ and $G$ are isomorphic as abstract groups. If $A$ is commutative, we supply $U$ and $H$ with the discrete topologies. If $A$ is finite dimensional, we supply $U$ and $H$ with the topologies that they inherit as subsets of the normed space $\cp$. Then $H$ is an amenable locally compact Hausdorff topological group in both cases:
\begin{itemize}
\item If $A$ is commutative, $H$ is clearly a locally compact Hausdorff topological group. Since $U$ is abelian, it is an amenable locally compact Hausdorff topological group; see e.g.\ \cite[Proposition 12.2]{pier}. Since $U$ and $H/U$ (which is isomorphic to $G$ as a locally compact Hausdorff topological group) are both amenable locally compact Hausdorff topological groups, \cite[Proposition 13.4]{pier} shows that $H$ is an amenable locally compact Hausdorff topological group.
\item  If $A$ is finite dimensional, then $U$ is compact, and it is easy to see that $H$ is a locally compact Hausdorff topological group. Since $U$ is compact, $U$ is an amenable locally compact Hausdorff topological group. As in the case where $A$ is commutative, the combination of the amenability of $G$ and \cite[Proposition 13.4]{pier} shows that $H$ is an amenable locally compact Hausdorff topological group.
\end{itemize}

After these preparations, let $E$ be a Banach $\cp$-bimodule, and let $D: \cp\to E^\ast$ be a bounded derivation. Let $\phi$ denote the restriction of $D$ to $H$. The restrictions of the left and right actions of $\cp$ to $H$ make $E$ and $E^\ast$ into $H$-bimodules, and we want to apply Proposition~\ref{prop:cohomology} to $H$, $\phi$, $E$, and $E^\ast$. We verify the conditions under part~(2) thereof:
\begin{itemize}
\item Condition~(a) is trivially met when $A$ is commutative, since $H$ has the discrete topology in that case. If $A$ is finite dimensional, it is satisfied because in that case $H$ has the topology that it inherits from $\cp$.
\item
Since $\Vert h\Vert =1$ for all $h\in H$, we see from the very definition of a Banach  $\cp$-bimodule that condition~(b) is satisfied.
\item
If $A$ is commutative, then condition (c) is again trivially met. If $A$ is finite dimensional, then $H$ has the topology that it inherits from $\cp$, so that $\phi$ is even continuous when $E^*$ carries its norm topology.
\item
It is evident from the properties of $D$ that condition (d) is satisfied.
\item
Since $\Vert h\Vert =1$ for all $h\in H$, the boundedness of $D$ implies that condition (e) is satisfied.
\end{itemize}

As we had already observed, $H$ is an amenable locally compact Hausdorff topological group, and we can now apply Proposition~\ref{prop:cohomology} to conclude that there exists $x^\ast\in E^\ast$ such that
$$
\phi(h)=h\cdot x^\ast - \ x^\ast\cdot h\quad
$$
for all $h\in H$. Hence $D$ and the inner derivation on $\cp$ that corresponds to $x^\ast$ agree on $H$. By continuity, they also agree on the closed linear span of $H$. Since Lemma~\ref{lem:key} shows that this the whole algebra, we conclude that $\cp$ is amenable.
\end{proof}

\begin{remark}
As a thought experiment, one could attempt to prove along similar lines that the Beurling algebras in \cite[Theorem~2.4]{badecurtisdales} are amenable. The proof must break down for the algebras in  \cite[Theorem~2.4.(ii)]{badecurtisdales}, since these are known not to be amenable. Indeed it does: the corresponding group $H=\{\,z\delta_n: z\in\mathbb T,\,n\in\mathbb Z\,\}$ inside these algebras is then no longer norm bounded, and this implies that the conditions under (b) and (e) in Proposition~\ref{prop:cohomology} are not necessarily satisfied. The fact that these algebras are actually not amenable, implies that at least one of these conditions must be violated for at least one Banach bimodule over these algebras.
\end{remark}

As a by-product, we obtain the amenability of the crossed product \Cstar-algebra $A\rtimes_\alpha G$ that is associated with $(A,G,\alpha)$. Naturally, this is only a very special instance of the general theorem on amenable crossed product \Cstar-algebras discussed in the Introduction, but it is still illustrative how this can be inferred from Theorem~\ref{thm:main_result} and a general principle for amenable Banach algebras, rather than working with nuclearity of \Cstar-algebras.

\begin{corollary}\label{cor:envelope}
Let $(A,G,\alpha)$ be a \Cstar-dynamical system where $G$ is a discrete amenable group and $A$ is a commutative or finite dimensional \Cstar-algebra. Then $A\rtimes_\alpha G$ is an amenable Banach algebra.
\end{corollary}

\begin{proof}
The \Cstar-algebra $A\rtimes_\alpha G$ is the enveloping \Cstar-algebra of the involutive Banach algebra $\cp$. By the very construction of such an enveloping \Cstar-algebra, there is a continuous (even: contractive) homomorphism of $\cp$ into $A\rtimes_\alpha G$ with dense range. It is a general principle for Banach algebras (see \cite[Proposition~2.3.1]{runde}) that the amenability of $\cp$, as asserted in Theorem~\ref{thm:main_result}, then implies that $A\rtimes_\alpha G$ is also amenable.
\end{proof}

\begin{remark}
The argument in the proof of Corollary~\ref{cor:envelope} also shows that Johnson's theorem on the amenability of $\mathrm{L}^1(G)$ implies that the group \Cstar-algebra $\Cstar(G)$ of an amenable locally compact Hausdorff topological group is amenable.
\end{remark}

\section{Perspectives}\label{sec:perspectives}

In this section, we give a few thoughts on further developments.

The first issue that we discuss is the motivation for the restriction to the case of discrete $G$ and commutative or finite dimensional $A$ in Theorem~\ref{thm:main_result}. The statement that the twisted convolution algebra $\textup{L}^1(G,A;\alpha)$ is an amenable Banach algebra makes sense for any \Cstar-dynamical system $(A,G,\alpha)$ where $G$ is an arbitrary locally compact Hausdorff topological group $G$ and $A$ is an arbitrary amenable \Cstar-algebra, and may, so we think, well be true. The reason that this more general problem is not taken up in the present paper is the following. Under the restrictive assumptions as described, the group $H\simeq U\rtimes_\alpha G$ in the proof of Theorem~\ref{thm:main_result} is relatively easily seen to be an amenable locally compact Hausdorff topological group, and the closed linear span of $H$ equals $\cp$.  Both these facts are instrumental in the proof of Theorem~\ref{thm:main_result}. In a more general context, this convenient setting is no longer present.  One can then not even see $G$ as a subgroup of $\textup{L}^1(G,A;\alpha)$, let alone work with an analogue of $H$ as a subset of $\textup{L}^1(G,A;\alpha)$. The whole structure of the proof of Theorem~\ref{thm:main_result} is simply no longer applicable.

Yet not all need be lost, and we shall now indicate why we think that hope for more is still justified. This will also make clear why we have restricted the validity of a presumed general principle, as mentioned in the Introduction, to the case where $A$ is, in fact, an amenable \Cstar-algebra, and not an arbitrary amenable Banach algebra.  For this, we recall from \cite[Theorem~2]{patersonpaper} that, if $A$ is a unital \Cstar-algebra with unitary group $U$, then $U$ is a topological group when supplied with the inherited weak topology of the Banach space $A$. Moreover, $U$ is amenable (in our sense that there exists a left invariant mean on the space of bounded right uniformly continuous complex functions on $U$) precisely when $A$ is amenable.  Furthermore, we see from \cite[Proposition~6.4]{DdJW} that there are natural homomorphisms from $G$ and $A$ into the left centralizer algebra $\mathcal M_{\textup l}(\textup{L}^1(G,A;\alpha))$ of $\textup{L}^1(G,A;\alpha)$.  The existence of these homomorphisms and their properties imply that $\mathcal M_{\textup l}(\textup{L}^1(G,A;\alpha))$ contains a group that is a homomorphic image of $U\rtimes_\alpha G$.  With these ingredients, one could now attempt to incorporate modifications of the ideas in the proof of Theorem~\ref{thm:main_result} into Johnson's proof of the amenability of $\textup L^1(G)$ for general amenable $G$\textemdash which proceeds via the left centralizer algebra of $\textup{L}^1(G)$)\textemdash and tackle the case of general $G$, $A$, and $\alpha$.  Needless to say, if such an approach is feasible, this will be considerably more technically demanding than the proof of Theorem~\ref{thm:main_result}.

The second issue that naturally comes to mind is the analogue of the `if'-part of Johnson's theorem. This part is due to Ringrose: if $\textup{L}^1(G)$ is amenable, then so is $G$; see \cite[Theorem~2.5]{johnson}.  We currently have no results in this direction, but one would hope that, for a \Cstar-dynamical system $(A,G,\alpha)$, where $G$ is an arbitrary locally compact Hausdorff topological group and $A$ is an arbitrary \Cstar-algebra, the amenability of $\textup{L}^1(G,A;\alpha)$ implies the amenability of both $G$ and $A$.

Ideally, then, the amenability of $\textup{L}^1(G,A;\alpha)$ would be equivalent to the amenability of both $G$ and $A$. More generally, for a \Cstar-dynamical system $(A,G,\alpha)$ and a non-empty uniformly bounded class $\mathcal R$ of continuous covariant representation thereof (see \cite[Definition~3.1]{DdJW} for this notion), one could hope that, under suitable assumptions on $\mathcal R$,  the amenability of the crossed product Banach algebra $(A\rtimes_\alpha G)^{\mathcal R}$ (see \cite[Definition~3.2]{DdJW} for its definition) is equivalent to the amenability of both $G$ and $A$. Whereas this aesthetically pleasing general statement is admittedly perhaps already in the realm of speculation, it seems well worth investigating these matters more closely, since, no matter the outcome, this will lead to a better understanding of the relation between amenability and the crossed product construction. We hope to be able to report on this in the future.

\subsection*{Acknowledgements}
We thank Jun Tomiyama for helpful comments.
The last author was partially funded by FCT/Portugal through project UID/MAT/04459/2013.

\end{document}